\documentclass[12pt,twoside,reqno]{amsart}
\usepackage[margin=1.0 in]{geometry}
\usepackage{enumitem}

\newtheorem{theorem}{Theorem}[section]
\newtheorem{lemma}[theorem]{Lemma}
\newtheorem{conjecture}[theorem]{Conjecture}
\newtheorem{obs}[theorem]{Observation}
\newtheorem{prop}[theorem]{Proposition}

\newtheorem{notation}[theorem]{Notation}
\newtheorem{fact}[theorem]{Fact}

\newtheorem{example}[theorem]{Example}
\newtheorem{remark}[theorem]{Remark}
\newtheorem{corollary}[theorem]{Corollary}

\newcommand{\F}{\mathbb{F}}
\newcommand{\I}{\mathcal{I}}
\newcommand{\J}{\mathcal{J}}
\newcommand{\K}{\mathcal{K}}
\newcommand{\ml}{\mathcal{L}}

\usepackage{relsize}
\makeatletter
\newcommand*{\rom}[1]{\expandafter\@slowromancap\romannumeral #1@}
\makeatother

\makeatletter
\def\imod#1{\allowbreak\mkern5mu({\operator@font mod}\,\,#1)}
\makeatother
\allowdisplaybreaks

\makeatletter
\@namedef{subjclassname@2020}{%
  \textup{2020} Mathematics Subject Classification}
\makeatother

\usepackage{microtype} 

 \usepackage{hyperref}
\hypersetup{
     colorlinks=true,
     citecolor=black,
     linkcolor=black,      
     urlcolor=black
     }

\title[On the WLP for ideals generated by powers of general linear forms]{On the Weak Lefschetz Property for ideals generated by powers of general linear forms}

\author[Matthew D. Booth]{Matthew D. Booth}
\author[Pankaj Singh]{Pankaj Singh}
\author[Adela Vraciu]{Adela Vraciu}

\address{Department of Mathematics, University of South Carolina, 
Columbia, SC 29208}
\email{mdbooth@email.sc.edu}
\email{pksingh@email.sc.edu}
\email{vraciu@math.sc.edu}

\date{\today}

\begin{document}

\begin{abstract} 
We provide a description of initial ideals for almost complete intersections generated by powers of general linear forms and prove that WLP in a fixed degree $d$ holds when the number of variables $n$ is sufficiently large compared to $d$. In particular, we show that if $n\geq 3d-2$ then WLP holds for the ideal generated by squares at the degree $d$ spot and for $n\ge \frac{3d-3}{2}$ WLP holds for ideal generated by cubes at the degree $d$ spot. Finally, we prove that WLP fails for the ideal generated by squares when $n< 3d -2$ at the $d$th spot by finding an explicit element in the kernel of the multiplication by a general linear form. This shows that our bound on $n$ is sharp in the case of the squares.
\end{abstract}

\maketitle
\tableofcontents

\section{Introduction}

Let $\F$ be a field of characteristic zero, $P_n=\F[x_1,\dots,x_n]$, and $A=P_n/I$ where $I$ is an Artinian monomial complete intersection. Given a general linear form $L\in A_1$, it was first proven in \cite{Stanley-80} (see also \cite{Watanabe}) that the homomorphism $\times L^p:A_{d-p}\rightarrow A_{d}$ has maximal rank (i.e. is either injective or surjective) for all positive integers $d$ and $p$. Any standard graded Artinian algebra satisfying this condition is said to have the \textit{Strong Lefschetz Property} (\textit{SLP}). An interesting and fruitful offshoot is the following weakening of the SLP: if $A$ is a standard graded Artinian algebra, then we say $A$ has the \textit{Weak Lefschetz Property} (\textit{WLP}) if the homomorphism $\times L^p:A_{d-p}\rightarrow A_{p}$ has maximal rank for $p=1$ and all positive integers $d$. Since its introduction, the Weak Lefschetz Property has been examined from a variety of angles. An excellent overview of its development can be found in \cite{Migliore-Nagel}.

One expects the Weak Lefschetz Property to hold for ``most'' algebras, so it is interesting to find classes of Artinian algebras where it fails. It is shown in \cite{Har-Mig-Nag-Wat} that all homogeneous ideals in $\F[x_1,x_2]$ enjoy SLP (hence also WLP) and that all Artinian complete intersections in $\F[x_1,x_2,x_3]$ enjoy WLP. Moreover, it is conjectured (see, for instance, the last few lines in \cite{Rei-Rob-Roi}) that any complete intersection over a field of characteristic zero has WLP, but this question remains open outside some special cases. Indeed, as the number of variables increases, deciding whether a given Artinian algebra has WLP is generally a rather difficult question. Section 2 of \cite{Mig-Mir-Nag-11} provides some tools to limit the number of required verifications, but even then only a handful of defining ideal formats are subsequently dealt with there. One of the results from \cite{Mig-Mir-Nag-11} that is relevant to the work in this paper is Proposition (2.1) part (b): given a level Artinian standard graded algebra, if injectivity holds for multiplication by a general linear form at some degree $d_0$, then injectivity also holds for all degrees $d\leq d_0$. (There is also a similar statement regarding surjectivity). This means that when we encounter an Artinian algebra which fails WLP, it is interesting to ask for the smallest degree where injectivity fails.

A case which has received particular attention recently is that of an ideal generated by $n+1$ general linear forms raised to a uniform power, which is to say $I=(\ell_1^p,\dots,\ell_{n+1}^p)$. We note that via a change of coordinates, the $n+1$ general linear forms can be taken to be $x_1, \ldots, x_n, x_1+\cdots + x_n$. It was conjectured in \cite{Har-Sch-Sec} that $A=\F[x_1,\dots,x_n]/(\ell_1^p,\dots,\ell_{n+1}^p)$ fails to have WLP if $n$ is sufficiently large. The case $n=3$ was settled by Theorem 2.4 in \cite{Schenck-Seceleanu}, where it is shown that WLP holds for any Artinian quotient of $\F[x_{1},x_{2},x_{3}]$ generated by powers of linear forms. In another direction, the case $p=2$ was settled by Theorem 2.12 in \cite{MiroRoig}, which states that $A$ enjoys WLP precisely when either $2\leq n\leq 5$ or $n=7$ holds. A partial general resolution to this conjecture was provided in Theorem 6.1 of \cite{Mig-Mir-Nag-12} (when $n\geq 6$ is even), and a complete classification followed in Theorem 1.1 of \cite{Boij-Lundqvist}. The content of the latter theorem is that $A$ as defined above enjoys WLP precisely when one of $n\leq 3$, $p=1$, or $(n,p)\in\big\{(4,2),\,(5,2),\,(5,3),\,(7,2)\big\}$ holds.

Given a general linear form $L$, showing failure of the WLP for an algebra comes down to finding a degree where the $\times L$ map fails to be injective or surjective. The results of \cite{Mig-Mir-Nag-12} and \cite{Boij-Lundqvist} show failure of surjectivity in some degree. While these results settle the question of whether WLP holds for such algebras, the present paper represents a sort of counterpoint in the sense that our results concern failure of injectivity when multiplying by a general linear form. In particular, we provide a lower bound on the earliest degree for which WLP can fail, and we improve this bound when the defining ideal is generated by squares or cubes. We conclude by showing that the bound we obtain in the square case is sharp.

\smallskip

We set some notation: $P_n=\F[x_1,\dots,x_n]$, and  $I_{n,a}=\big(x_1^a,\dots,x_n^a,(x_1+\cdots+x_n)^a\big)$. For a graded ring $R$  the notation $R_d$ indicates the $d$th graded piece of $R$. Similarly, if $I$ is a homogeneous ideal in a graded ring, $I_d$ is the $d$th graded piece of $I$.

Our general result for the lower bound is the following:

\begin{theorem}
Let $n, a, d$ be positive integers with $d\ge a\ge 2$, and define $N(a,d)=\displaystyle \left\lceil\frac{2d}{a-1}\right\rceil$. If $\displaystyle n \ge N(a, d)+\frac{2d-1}{a-1}$, then the map given by multiplication by a general linear form from $(P_n/I_{n, a})_{d-1}$ to $(P_n/I_{n,a})_d$ is injective.
\end{theorem}

As previously mentioned, this bound is not sharp. When $a=2$ (i.e. the defining ideal is generated by squares), our bound gives $n\geq 4d-1$, but we will show in Theorem (\ref{injectivity for squares}) that $n\geq 3d-2$ is enough to guarantee injectivity from degree $d-1$ to degree $d$ in this case. Moreover, $n\geq 3d-2$ is sharp for an ideal generated by squares per Theorem (\ref{square injectivity failure}), whose content is the failure of injectivity from degree $d-1$ to degree $d$ when $n\leq 3d-3$. In a similar way, $a=3$ gives the bound $n\geq 2d$, but Theorem (\ref{inj for cubes}) shows $n\geq\frac{3d-3}{2}$ will suffice to guarantee injectivity from degree $d-1$ to degree $d$ when the defining ideal is generated by cubes.

The presence of the Weak Lefschetz Property is closely connected to the famous conjecture of Fr\"{o}berg predicting the Hilbert series of a quotient of a polynomial ring by general forms of specified degrees. More precisely, if $R=\F[x_1, \ldots, x_n]/(f_1, \ldots, f_r)$ with $f_1, \ldots, f_r$ general forms of respective degrees $d_1, \ldots, d_r$, Fr\"{o}berg's conjecture states that 
\begin{equation}\label{froberg}
\mathrm{HS}(R, t)=\left[\frac{\prod_{k=1}^r (1-t^{d_k})}{(1-t)^n}\right]
\end{equation}
where $\mathrm{HS}(R, t)=\sum_{d=0}^{\infty} \mathrm{dim}_{\F} (R_d)t^d$ and $[\ \ \ \  ]$ denotes truncation before the first non-positive term.
We will say that (\ref{froberg}) holds for $R$ up to degree $D$ if 
\begin{equation}
\big[\mathrm{HS}(R, t)\big]_D=\left[\frac{\prod_{k=1}^r (1-t^{d_k})}{(1-t)^n}\right]_D
\end{equation}
where $[\ \ \ ]_D$ means truncation at the $t^D$ term or before the first non-positive term, whichever is sooner.

It is shown in \cite{Froberg} that the inequality $\ge $ in (\ref{froberg}) is always true for any $f_1, \ldots, f_r$ of degrees $d_1, \ldots, d_r$, and the existence of one choice of such $f_1, \ldots f_r$ where equality holds implies that equality holds for general forms. 
Although the conjecture is widely believed to be true, there are only a few instances in which it is known. In particular, the Conjecture (\ref{froberg}) is true when $r=n+1$ (see \cite{Stanley-78}) as well as when $(d,n)= (2, 12), (2, 13), (3, 9)$ (see \cite{froberg2024}). We refer the interested reader to the survey \cite{Froberg-Lundqvist} for more information.\vspace{12pt}

\indent The following observation will allow us to derive some cases where the Fr\"oberg conjecture holds up to a certain degree as corollaries of our results about injectivity.
\begin{obs}\label{Froberg}
Let $R=P_n/(f_1, \ldots, f_r)$ where $f_1, \ldots, f_r$ are forms of degrees $d_1, \ldots, d_r$, and let $L$ be a linear form. Assume that (\ref{froberg}) holds for $R$ up to degree $D$, for some $D$ with $R_D\ne 0$.

Then the map given by multiplication by $L:R_{d-1}\to R_d$ has maximal rank for all $d\le D$ if and only if (\ref{froberg}) holds for $R/(L)$ up to degree $D$, where we view $R/(L)$ as the quotient of $P_{n-1}$ by $r$ forms of degrees $d_1, \ldots, d_r$.
\end{obs}

\begin{proof}
Since $R/(L)$ is the cokernel of this map, injectivity of the map given by multiplication by $L: R_{d-1}\to R_d$ is equivalent to 
$$
\mathrm{dim}_{\F}((R/(L))_d)=\mathrm{max}\{\mathrm{dim}_{\F}(R_d)-\mathrm{dim}_{\F}(R_{d-1}), 0\}.
$$
Therefore, injectivity for all $d\le D$ is equivalent to 
$$
\big[\mathrm{HS}(R/(L), t)\big]_D=\big[(1-t)\mathrm{HS}(R, t)\big]_D.
$$
\end{proof}

In particular, since WLP holds for monomial complete intersections by \cite{Stanley-80}, we have the following:
\begin{remark}\label{HF of ACI}
$$
\mathrm{dim}\left(\left(\frac{P_n}{I_{n, a}}\right)_d\right)=\mathrm{max}\{\mathrm{dim}\left(\left(\frac{P_{n+1}}{(x_1^a, \ldots, x_{n+1}^a)}\right)_d\right)-\mathrm{dim}\left(\left(\frac{P_{n+1}}{(x_1^a, \ldots, x_{n+1}^a)}\right)_{d-1}\right), 0\}.
$$
\end{remark}

This disparity between the general bound of Theorem 1.1 and the bounds obtained in the cases of $a=2$ and $a=3$ is due to our understanding of the initial ideal in each case. Indeed, the proofs of our results concerning injectivity share a similar structure in which we first find the initial ideal with respect to the reverse lexicographic order. The cases $a=2$ and $a=3$ permit explicit descriptions per Theorems (\ref{initial2}) and (\ref{initial3}), respectively; the corresponding description in Theorem (\ref{general bound}) for an arbitrary exponent $a$ is less precise. We shall outline this structure in more detail after recording two results which are critical for our arguments. The first is due to Wiebe (and is a quick consequence of Theorem 1.1 in \cite{Conca}):

\begin{prop}[\cite{Wiebe}, Proposition 2.9]\label{wiebe}
Let $I\subseteq P_n$ be an $\mathfrak{m}$-primary graded ideal where $\mathfrak{m}=(x_1,\dots,x_n)$, and fix any term order $\sigma$. If $P_n/\text{in}_{\sigma}(I)$ has WLP (resp. SLP), then $P_n/I$ also enjoys WLP (resp. SLP).
\end{prop}

Because an initial ideal is a monomial ideal, this result pairs well with our second vital tool due to Eddings and Vraciu:

\begin{lemma}[\cite{Eddings-Vraciu}, Lemma 3.1]\label{ayden}
Let $I\subseteq P_n$ be a monomial ideal, and let $\ell=\sum_{i=1}^n x_i$. Let $Q=\sum_{\mu } a_{\mu } \mu \in P_n$ be a form of degree $d-1$, where the summation is over all monomials $\mu$ of degree $d-1$.  Assume $\ell Q\in I$. Then for every monomial $M$ of degree $2d-1$, if $M\notin I$ then $a_{\mu }=0$ for every $\mu $ of degree $d-1$ such that $\mu $ divides $M$.
\end{lemma}

The outline of Sections 2, 3, and 4 can now be provided. After fixing the notation, the first step will be a description of the initial ideal with respect to the reverse lexicographic order. Then we use the Lemma (\ref{ayden}) to prove that injectivity holds for the initial ideal, whence the conclusion follows from Proposition (\ref{wiebe}). Section 5 is of an altogether different nature, constructing a nonzero element in the kernel of multiplication by a general linear form when the defining ideal is generated by squares.

We next record two more standard results that will be needed at various points. Their statements have been paraphrased from the source material to suit the context of this paper.

\begin{theorem}[\cite{Rei-Rob-Roi}, Theorem 1]\label{CI H-func & soc deg}
The Hilbert function of a complete intersection is symmetric about half the socle degree. Moreover, the Hilbert function is increasing until half the socle degree and is decreasing afterward.
\end{theorem}

\begin{fact}[\cite{Eisenbud}, Proposition 15.12]\label{revlex props}
Consider $P_n=\F[x_1,\dots,x_n]$ with reverse lexicographic order, and let $I$ be a homogeneous ideal of $P_n$. Then the initial ideal enjoys the following two properties: $\mathrm{in}\big(I+(x_n)\big)=\mathrm{in}(I)+(x_n)$ and $\big(\mathrm{in}(I):_{P_n}x_n\big)=\mathrm{in}(I:_{P_n} x_n)$.
\end{fact}

\begin{remark}
\normalfont{Shortly after the completion of this paper, we were made aware of \cite{lundqvist} in which Kling, Lundqvist, Mohammadi, Orth, and S\'aenz-de-Cabez\'on independently provided a description of the initial ideal of $(x_1^2,\ldots, x_n^2, (x_1+ \ldots + x_n)^2)$ by computing the Gr\"{o}bner basis for it. In particular, their methodology differs significantly from our proof of Theorem (\ref{initial2}}).
\end{remark}

{\textit{Acknowledgments.}} To conclude the introduction, the authors would like to thank the anonymous referee for providing several useful comments and especially for bringing the results in \cite{CriQ-Lun-Nen} to our attention. The authors also wish to acknowledge the computer algebra system \texttt{Macaulay2} \cite{M2}, which provided valuable assistance in realizing and formulating the results of this paper. Finally, we thank Joshua Cooper for helpful conversations related to the subject of this paper.

\section{A lower bound on the number of variables such that injectivity holds}


The main result of this section is Theorem (\ref{injectivity bound}), in which 
we find an explicit bound depending on $a$ and $d$ (which is linear in $d$ for a fixed exponent $a$) such that the map given by multiplication by a general linear form from $(P_n/I_{n, a})_{d-1}$ to $(P_n/I_{n, a})_d$ is injective whenever $n$ is larger than the bound.

The key ingredient is the following description of the initial ideal of $I_{n, a}$:

\begin{theorem}\label{general bound}
Let $n, a, d$ be positive integers with $d\ge a\ge 2$. Let $$\displaystyle N(a, d):=\left\lceil \frac{2d}{a-1}\right\rceil .$$ For a fixed degree $d$ and all $n\ge N(a, d)$, the minimal generators of degree $\le d$ of the initial ideal of 
$I_{n, a}$ in reverse lexicographic order consist of $x_1^a, \ldots, x_n^a$, and monomials that only involve the variables $x_1, \ldots, x_{N(a, d)}$.
\end{theorem}

In the following proof, the notation $P_{n,d}$ refers to the $d^{\text{th}}$ graded piece of $P_{n}$.

\begin{proof}
Let $K\subseteq \F[x_1, \ldots, x_{N(a, d)}]$ denote the initial ideal of $I_{N(a, d), a}$ and let $K_n$ denote the ideal of $\F[x_1, \ldots,x_n]$ generated by the image of $K$ under the natural inclusion for $n\ge N(a, d)$. The containment $K_n+(x_1^a, \ldots, x_n^a)\subseteq \mathrm{in}(I_{n, a})$ is clear from the definition of $K_n$.

Since $I_{n, a}$ and $\mathrm{in}(I_{n, a})$ have the same Hilbert function, the conclusion will follow once we prove that 
\begin{equation}\label{equ}
\mathrm{dim}\left(\frac{P_{n,d}}{(K_n+(x_{N(a, d)+1}^a, \ldots, x_n^a))_d}\right)=\mathrm{dim} \left(\frac{P_{n,d}}{(I_{n, a})_d}\right)
\end{equation}
For a fixed $a$, let $E(n, d)$ denote the left-hand side of (\ref{equ}) and $F(n, d)$ denote the right-hand side.

The two sides of equation (\ref{equ}) are the same for $n=N(a, d)$  by definition. We claim that both satisfy the same recurrence relation as functions of $n$ (for fixed $a$ and $d$) for $n\ge N(a, d)$.

Since the generators of $K_n$ only involve the variables $x_1, \ldots, x_{N(a, d)}$, we have
$$\frac{P_n}{K_n+(x_{N(a, d)+1}^a, \ldots, x_n^a)}= \frac{\F[x_1, \ldots, x_{n-1}]}{K_n+(x_{N(a, d)+1}^a, \ldots, x_{n-1}^a)}\otimes_{\F} \frac{\F[x_n]}{(x_n^a)}.
$$
Using a standard property of the tensor product, for each $d\ge 0$ we have
$$\left(\frac{P_n}{K_n+(x_{N(a, d)+1}^a, \ldots, x_n^a)}\right)_d= \bigoplus_{i=0}^d \left(\frac{\F[x_1, \ldots, x_{n-1}]}{K_n+(x_{N(a, d)+1}^a, \ldots, x_{n-1}^a)}\right)_i \otimes_{\F} \left(\frac{\F[x_n]}{(x_n^a)}\right)_{d-i}.
$$
Since the dimension of $\displaystyle{\frac{\F[x_n]}{(x_n^a)}}$ as an $\F$-vector space is $1$ in each degree less than $a$ and $0$ otherwise, we obtain the recurrence relation
$$
E(n, d)=E(n-1, d)+E(n-1, d-1)+\cdots + E(n-1, d-a+1).
$$
For the right-hand side of (\ref{equ}), an appeal to Remark (\ref{HF of ACI}) gives the following:
$$
F(n, d)=\mathrm{max}\{P(n+1, d)-P(n+1, d-1), 0\},
$$
where 
$$P(n+1, d)=\mathrm{dim}\left(\frac{\F[x_1, \ldots, x_{n+1}]_d}{(x_1^a, \ldots, x_{n+1}^a)_d} \right).
$$
The socle degree of the complete intersection $\displaystyle\frac{\F[x_1, \ldots, x_{n+1}]}{(x_1^a,\dots,x_{n+1}^a)}$ is $(a-1)(n+1)$, so by Theorem (\ref{CI H-func & soc deg}) we have
\begin{equation}\label{inequality}
P(n+1, d)\ge P(n+1, d-1)\quad\Longleftrightarrow\quad d  \le \frac{(a-1)(n+1)}{2}\quad\Longleftrightarrow\quad n\ge \frac{2d}{a-1}-1.
\end{equation}
We have the recurrence relation 
$$
P(n+1, d)=P(n, d)+P(n, d-1)+\cdots + P(n, d-a+1).
$$
When (\ref{inequality}) holds, it follows that 
$$
F(n, d)=P(n+1, d)-P(n+1, d-1)=\sum_{d'=d-a+1}^d \big(P(n, d')-P(n, d'-1)\big).
$$
On the other hand, when $\displaystyle n\ge \frac{2d}{a-1}$, we have 
$P(n, d')-P(n, d'-1)=F(n-1, d')
$ for all $d'\le d$ (from (\ref{inequality}) applied to $n-1$ in the role of $n$), and thus
$$
F(n, d)=F(n-1, d) +F(n-1, d-1)+\cdots + F(n-1, d-a+1).
$$

Since $E(n, d)$ and $F(n, d)$ satisfy the same recurrence relation and are equal for $n=N(a, d)$, it follows that they continue to be equal for all $n\ge N(a, d)$.
\end{proof}

\begin{theorem}\label{injectivity bound}
    Let $n, a, d$ be positive integers with $d\ge a\ge 2$. If 
    $\displaystyle n \ge N(a, d)+\frac{2d-1}{a-1}$, then the map given by multiplication by a general linear form from $(P_n/I_{n, a})_{d-1}$ to $(P_n/I_{n,a})_d$ is injective.
\end{theorem}

\begin{proof}
By Proposition (\ref{wiebe}), it suffices to prove injectivity when $I_{n,a }$ is replaced by $\mathrm{in}(I_{n, a})$. Moreover, the injectivity from degree $d-1$ to degree $d$ is not affected by the generators of the ideal of higher degree. Therefore, we may replace $I_{n, a}$ with the ideal generated by the degree $d$ piece of $\mathrm{in}(I_{n, a})$.

Assume $Q=\sum a_{\mu } \mu $ is in the kernel of the map given by multiplication by $L=\sum_{i=1}^n x_i$ from $(P_n/\mathrm{in}(I_{n, a}))_{d-1}$ to $(P_n/\mathrm{in}(I_{n, a}))_d$, where the first summation is over all monomials $\mu $ of degree $d-1$. We need to prove that $a_{\mu }=0$ for all $\mu \notin \mathrm{in}(I_{n, a})$.

Let $\mu\notin \mathrm{in}(I_{n, a})$ be a monomial of degree $d-1$. Let $\mu' :=\mathrm{gcd}(\mu, x_{N(a, d)+1}^{a-1}\cdots  x_n^{a-1})$, and define 
$$
M:= \mu \cdot \frac{x_{N(a, d)+1}^{a-1}\cdots x_{n}^{a-1}}{\mu '}.
$$
Multiplying by any variable $x_j$ with $j>N_{a,d}$ does not affect membership in the ideal generated by the degree $d$ piece of $\mathrm{in}(I_{n, a})$; therefore, $M$ is not in this ideal.
We have 
$$
\mathrm{deg}(M)\ge (a-1)(n-N(a, d))\ge 2d-1,
$$
and now the result follows from Lemma (\ref{ayden}).
\end{proof}

As a corollary, we see that Fr\"oberg's conjecture holds for $n+2$ forms of degree $a$ up to $D:=\lfloor\frac{n(a-1)+1}{4}\rfloor$:
\begin{corollary}
Let $R=P_{n}/(x_{1}^{a},\dots,x_{n}^{a},\ell_{1}^{a},\ell_{2}^{a})$ where $\ell_{1}$ and $\ell_{2}$ are general linear forms. We have
\begin{equation*}
    \big[\mathrm{HS}(R,t)\big]_{D}=\left[\frac{(1-t^{a})^{n+2}}{(1-t)^{n}}\right]_{D}.
\end{equation*}
where $D:=\lfloor\frac{n(a-1)+1}{4}\rfloor$
\end{corollary}

\begin{proof}
This follows immediately from the preceding theorem and Observation (\ref{Froberg}).
\end{proof}

\section{Injectivity for ideals generated by squares}

\begin{notation}
    Let $n\geq 1$ be an integer. We fix the following:
    \begin{enumerate}[itemsep=5pt]
        \item $P_n$ denotes the polynomial ring $\F[x_1,\dots,x_n]$ where $\F$ is a field of characteristic $0$. In addition, the notation $P_{n,d}$ will refer to the $d^{\text{th}}$ graded piece of $P_n$.
        \item The quotient ring $R_n:=\F[x_1,\dots,x_{n+1}]/(x_1^2,\dots,x_{n+1}^2)$ and the linear form $\ell:=\sum_{i=1}^{n+1}x_i$. Note that $R_n$ has WLP according to \cite{Stanley-80}.
        \item The ideal $I_n:=\big(x_1^2,\dots,x_{n}^2, (x_1+\cdots+x_n)^2\big)$ in the ring $P_n$. In addition, $I_{n, d}$ will refer to the $d^{\mathrm{th}}$ graded piece of $I_n$, and $\mathrm{in}(I_n)$ is the initial ideal of $I_n$ with respect to the reverse lexicographic order.
    \end{enumerate}
\end{notation}

The main result of this section is Theorem (\ref{injectivity for squares}), which shows that if $n\ge 3d-2$, the map given by multiplication by a general linear form from $(P_n/I_n)_{d-1}$ to $(P_n/I_n)_d$ is injective. 

The key ingredient is the following description of the initial ideal of $I_n$.

\begin{theorem}\label{initial2}
The initial ideal $\mathrm{in}(I_n)$ with respect to the reverse lexicographic order is generated by $x_1^2, \ldots, x_n^2$, and all square-free monomials of degree $d$ in the variables $x_1, \ldots, x_{2d-2}$ for $2d-2\le n$.
\end{theorem}

Before giving the proof, we set up some notation. Let $D_k$ be the set of square-free monomials of degree $2$ in the variables $x_1, \ldots, x_{2k-2}$, for $2k-2\le n$, which are not divisible by any monomials of degree $k'$ in the variables $x_1, \ldots, x_{2k'-2}$ for any $k'<k$. For instance, the first few $D_k$ are given explicitly by
\begin{equation*}
    D_{2}=\{x_1x_2\},\qquad D_{3}=\{x_1x_3x_4,\,x_2x_3x_4\},\quad\text{and}
\end{equation*}
\begin{equation*}
    D_{4}=\{x_1x_3x_5x_6,\,x_2x_3x_5x_6,\,x_1x_4x_5x_6,\,x_2x_4x_5x_6,\,x_3x_4x_5x_6\}.
\end{equation*}
After proving the theorem, we will obtain Corollary (\ref{min gen set for squares}), which expresses the initial ideal in terms of the $D_k$. It will say, for example, that $\mathrm{in}(I_6)$ is minimally generated by $\{x_1^2,\dots,x_6^2\}\cup D_2\cup D_3\cup D_4$.

\begin{proof}
For all degrees $d$, let $K_{n, d}$ denote the vector space spanned by all the degree $d$ monomials that are multiples of a monomial in $D_k$ for some $k\le d$.

We claim that 
\begin{equation}\label{want}
\mathrm{dim}\left(\frac{P_{n, d}}{K_{n, d}}\right) = \mathrm{dim}\left(\frac{P_{n, d}}{I_{n, d}}\right)\end{equation}
 
Remark (\ref{HF of ACI}) gives 
\begin{equation}\label{dim}
\mathrm{dim}\left(\frac{P_{n,d}}{I_{n,d}}\right) = \mathrm{dim}\left( \left(\frac{R_n}{(\ell)}\right)_d\right)
=\mathrm{max}\{\mathrm{dim}((R_n)_d) - \mathrm{dim}((R_n)_{d-1}), 0\}
\end{equation}
\begin{equation*}
=\mathrm{max}\left\{\left(\begin{array}{c} n+1\\ d \\ \end{array}\right) - \left(\begin{array}{c} n+1 \\ d-1 \\ \end{array}\right), 0\right\} = \mathrm{max}\left\{\left( \begin{array}{c} n \\ d \\ \end{array}\right) -\left(\begin{array}{c} n\\ d-2 \\ \end{array}\right), 0\right\}.
\end{equation*}
Note that this is equal to zero for $n\le 2d-2$. On the other hand, $K_{n, d}$ consists of all monomials of degree $d$ in $P_n$ for $n\le 2d-2$, and therefore the left-hand side of equation (\ref{want}) is also equal to zero for $n\le 2d-2$. Moreover, all monomials of degree $d$ in $P_n$ must belong to $\mathrm{in}(I_n)$ for $n\le 2d-2$, so we have, in fact, $(\mathrm{in}(I_n))_d=K_{n, d}$.

It remains to prove that equation (\ref{want}) holds for all $n\ge 2d-2$, and we proceed by induction on $d$.

For $d=2$:
$\mathrm{dim}(P_{n,2}/K_{n, 2})$ is the number of square-free monomials, excluding $x_1x_2$ (since $K_{n, 2}$ is the span of $x_1^2, \ldots, x_n^2, x_1x_2$), i.e. $\displaystyle \left(\begin{array}{c} n\\ 2 \\ \end{array}\right) -1 $; equation (\ref{dim}) gives the same result for $\mathrm{dim}(P_{n,2}/I_{n, 2})$.

Assume $d\ge 3$. Let $E(n, d):= \mathrm{dim}(P_{n,d}/K_{n, d})$ and $F(n, d):= \mathrm{dim}(P_{n,d}/I_{n,d})$. We need to prove that $E(n, d)=F(n, d)$ for all $n\ge 2d-2$. By the inductive hypothesis, we have $E(n, d-1)=F(n, d-1)$ for all $n \ge 2d-4$. Fixing $d$, we now induct $n$.

For $n=2d-2$, we have $E(2d-2, d)=F(2d-2, d)=0$. Let $n\ge 2d-1$.
We claim that
\begin{equation}\label{quad-rec}
    \begin{array}{ccl}
    E(n, d)& =& E(n-1, d-1)+E(n-1, d) \ \ \ \mathrm{and} \\  F(n, d)& =& F(n-1, d-1)+F(n-1, d).
    \end{array}
\end{equation}
For the first equality in claim (\ref{quad-rec}), note that $E(n, d)$ is the number of monomials of degree $d$ in $n$ variables that are not divisible by any monomial of degree $k$ in the variables $x_1, \ldots, x_{2k-2}$, for any $k\le d$. We partition this set into monomials that are divisible by $x_n$ and monomials that are not divisible by $x_n$. Those divisible by $x_n$ can be written as $x_nm$, with $m$ a monomial of degree $d-1$ in $n-1$ variables, not divisible by any monomial of degree $k$ in the variables $x_1, \ldots, x_{2k-2}$ for any $k\le d$. The number of such monomials is $E(n-1, d-1)$. The number of monomials that are not divisible by $x_n$ is $E(n-1, d)$, since we can view these as monomials in $n-1$ variables. 

The second equality in claim (\ref{quad-rec}) follows from (\ref{dim}) and the recurrence relation for binomial coefficients. 

The inductive step now follows from (\ref{quad-rec}), which concludes the proof of equation (\ref{want}).

For $n>2d-2$, let $\overline{I_n}$ and $\overline{\mathrm{in}(I_n)}$ denote the images of $I_n$ and $\mathrm{in}(I_n)$ in $P_n/(x_{2d-1}, \ldots, x_n)=P_{2d-2}$. Note that $\overline{I_n}=I_{2d-2}$. By Fact (\ref{revlex props}), we have $\mathrm{in}(\overline{I_n})=\overline{\mathrm{in}(I_n)}$. 
Recalling that we have already established $K_{2d-2, d}= (\mathrm{in}(I_{2d-2}))_d$, it now follows that $K_{n, d}\subseteq (\mathrm{in}(I_n))_d$. Coupled with (\ref{want}), we now conclude that $K_{n, d}=(\mathrm{in}(I_n))_d$.
\end{proof}

As an immediate consequence, we have the following description of a minimal set of generators for $\mathrm{in}(I_n)$:
\begin{corollary}\label{min gen set for squares}
A minimal set of generators for $\mathrm{in}(I_n)$ is $\displaystyle \{x_1^2, \ldots, x_n^2\} \cup \bigcup_{2k-2\le n} D_k$.
\end{corollary}

\begin{proof}
This follows directly from Theorem (\ref{initial2}).    
\end{proof}

Before proving the main result of this section, we will need two lemmas. They allow us to construct a multiple of a given monomial that avoids membership in the initial ideal, positioning us to use Lemma (\ref{ayden}). To illustrate their utility and the main idea of their proofs, we provide a concrete example.

\begin{example}\label{ex1}
\normalfont{Fix $d=7$ and $n=2d-2=12$. Given a monomial $\mu\in P_{12}$ of degree $d-2=5$ which is not in $\mathrm{in}(I_{12})$, our goal is to build a monomial $\widetilde{\mu}$ of degree $d-1=6$ which is a multiple of $\mu$ and is not in $\mathrm{in}(I_{12})$.

We choose $\mu=x_1x_2x_3x_4x_{12}$ first. Note that $\deg{(\mu)}=5=d-2$, and moreover $\mu\notin\mathrm{in}(I_{12})$ by Theorem (\ref{initial2}). In this case, define $\widetilde{\mu}=x_{11}\mu$. This satisfies $\deg{(\widetilde{\mu})}=\deg{(\mu)}+1$, and Theorem (\ref{initial2}) ensures $\widetilde{\mu}\notin\mathrm{in}(I_{12})$. More generally, this construction will work, provided that $\mu$ is not divisible by $x_{11}x_{12}$.

Now, let us take $\mu=x_1x_2x_3x_{11}x_{12}$. In this case, define $\mu'=x_1x_2x_3$ so that $\mu=x_{11}x_{12}\mu'$. Observe that because it is square-free, $\mu'$ may be viewed as an element in $P_{10}$. It follows from Theorem (\ref{initial2}) that $\big(\mathrm{in}(I_{12})\big)_{t}=\big(\mathrm{in}(I_{10})\big)_{t}$ whenever $t\leq d-1=6$. Since $\mu\notin\mathrm{in}(I_{12})$ and $\deg{(\mu')}=3<6$, we deduce that $\mu'\notin\mathrm{in}(I_{10})$. With this, we define $\widetilde{\mu'}=x_{10}\mu'$ to reach $\deg{(\widetilde{\mu'})}=\deg{(\mu')}+1$ and $\widetilde{\mu'}\notin\mathrm{in}(I_{10})$. Finally, define $\widetilde{\mu}=x_{11}x_{12}\widetilde{\mu'}$ to obtain a monomial with the desired properties. More generally, this construction works when $x_{11}x_{12}$ divides $\mu$.}
\end{example}

\begin{lemma}\label{degree}
Let $d\ge 2$ and $n=2d-2$. If $\mu\in P_n $ is a monomial with $\mathrm{deg}(\mu):=\delta \le d-2$ and $\mu \notin \mathrm{in}(I_n)$, then there exists $\widetilde{\mu}\in P_n$ a multiple of $\mu$ with $\mathrm{deg}(\widetilde{\mu})=\mathrm{deg}(\mu ) +1$ and $\widetilde{\mu}\notin \mathrm{in}(I_n)$.
\end{lemma}

\begin{proof}
We induct on $d$. The base case is $d=2$, in which case $\mathrm{deg}(\mu)=0$ and the statement is trivial.

Let $d>2$. The assumption that $\mu\notin \mathrm{in}(I_n)$ means that $\mu$ is square-free and is not divisible by any monomial of degree $k$ in the variables $x_1,\ldots,x_{2k-2}$ for any $k\le \delta$.
If there exists $j\in \{2d-3, 2d-2\}$ such that $\mu$ is not divisible by $x_j$, let $\widetilde{\mu}:=\mu x_j$. Observe that $\widetilde{\mu}$ is not divisible by any monomial of degree $k$ in the variables $x_1, \ldots, x_{2k-2}$ for any $k\le \delta +1$. Indeed, since $\delta +1 \le d-1$ and $2(\delta +1) -2 \le 2d-4$, multiplication by $x_{2d-3}, x_{2d-2}$ does not affect this property, and therefore $\widetilde{\mu}\notin \mathrm{in}(I_n)$.

Otherwise, say $\mu =x_{2d-2}x_{2d-3}\mu '$ with $\mathrm{deg}(\mu')=\delta -2$. Since $\mu$ is square-free, $\mu'$ does not involve the variables $x_{2d-2}, x_{2d-3}$, and therefore we can view $\mu'$ as a monomial in $P_{2d-4}$. Theorem (\ref{initial2}) implies that the square-free monomial generators of $\mathrm{in}(I_{2d-2})$ that have degree $\le d-1$ are the same as the square-free monomial generators of $\mathrm{in}(I_{2d-4})$ of degree $\le d-1$, so the assumption $\mu \notin \mathrm{in}(I_{2d-2})$ implies $\mu'\notin \mathrm{in}(I_{2d-4})$. Thus, we may apply the inductive hypothesis to $\mu'$ in the role of $\mu$, whence there is a monomial $\widetilde{\mu '}\in P_{2d-4}$ that is a multiple of $\mu'$ of degree $\mathrm{deg}(\mu')+1=\delta -1$ with $\widetilde{\mu'}\notin \mathrm{in}(I_{2d-4})$. We can now define $\widetilde{\mu}:=x_{2d-2}x_{2d-3}\widetilde{\mu'}\in P_{2d-2}$, which satisfies the required properties. 
\end{proof}

\begin{lemma}\label{degree2}
    Let $d\ge 2$ and $n=2d-2$. If $\mu\in P_n $ is a monomial with $\mathrm{deg}(\mu)\le d-2$ and $\mu \notin \mathrm{in}(I_n)$, then there exists $\widetilde{\mu}\in P_n$ a multiple of $\mu$ with $\mathrm{deg}(\widetilde{\mu})=d-1$ and $\widetilde{\mu}\notin \mathrm{in}(I_n)$.
\end{lemma}

\begin{proof}
This follows by a repeated application of Lemma (\ref{degree}).
\end{proof}

We are now ready to prove the main result of this section.

\begin{theorem}\label{injectivity for squares}
Let $d$ be a positive integer. If $3d-2\le n$, then multiplication by a general linear form from $(P_n/I_n)_{d-1}$ to $(P_n/I_n)_d$ is injective.
\end{theorem}

\begin{proof}
By Proposition (\ref{wiebe}), it suffices to prove injectivity when $I_n$ is replaced by $\mathrm{in}(I_n)$. In other words, we need to prove that $\times L: (P_n/\mathrm{in}(I_n))_{d-1} \to (P_n/\mathrm{in}(I_n))_d$ is injective, where $L=\sum_{i=1}^n x_i$. Since this map is not affected by the higher degree generators of $\mathrm{in}(I_n)$, we may replace $\mathrm{in}(I_n)$ by the ideal generated by the elements of $\mathrm{in}(I_n)$ of degree $\le d$. Denote this ideal by $K_{n,d}$. Assume there is $Q=\sum_{\mu} a_{\mu} \mu \in P_{d-1}$ such that the image of $Q$ in $P/K_{n,d}$ is in the kernel of the multiplication by $L$ map. We need to prove that $a_{\mu }=0$ for every monomial $\mu $ of degree $d-1$ such that $\mu \notin K_{n, d}$. 
This will follow from Lemma (\ref{ayden}) provided that we can construct a monomial $M\in P_n$ such that $M$ is a multiple of $\mu$, $\mathrm{deg}(M)\ge 2d-1$, and $M\notin K_{n, d}$.

If all of $x_{2d-1}, \ldots, x_{3d-2}$ do not divide $\mu$, we can take
$M:=\mu x_{2d-1} \cdots x_{3d-2}$. $M$ is a square-free monomial of degree $2d-1$, and is not divisible by any monomials of degree $k$ in the variables $x_1,\ldots, x_{2k-2}$ for any $k\le d$, and therefore $M\notin K_{n, d}$.

Otherwise, say that $\mu = x_{j_1}\cdots x_{j_s} \mu'$ where $j_1, \ldots, j_s\in \{2d-1, \ldots, ,3d-2\}$ and $\mu'$ only contains variables up to $x_{2d-2}$. Applying Lemma (\ref{degree2}) to $\mu'$ in the role of $\mu$ gives us a monomial $\widetilde{\mu}$ of degree $d-1$ in the variables up to $x_{2d-2}$ which is a multiple of $\mu'$ and is not in $K_{n, d}$. We can now take $M:=\widetilde{\mu} x_{2d-1}\cdots x_{3d-2}$.
\end{proof}

It was pointed out to us by the referee that Theorem (\ref{injectivity for squares}) can be derived as a corollary of the results in \cite{CriQ-Lun-Nen}. More precisely, Theorem 8 in \cite{CriQ-Lun-Nen} shows that 
\begin{equation}\label{inqq1}
\dim_{\F}\left(\left(\frac{P_{n}}{(x_1^2, \dots, x_n^2, \ell_1^2,\ell_2^2)}\right)_{d}\right)\leq a_{n,d}
\end{equation}
where $a_{n, d}$ is a quantity defined as the number of lattice paths of a certain type. Moreover, Proposition 7 in \cite{CriQ-Lun-Nen} shows that for $\displaystyle d\le \left\lfloor \frac{n}{3}\right\rfloor +1$, 
$$
a_{n, d}=\left(\begin{array}{c} n\\ d\\ \end{array}\right) -2\left(\begin{array}{c} n\\ d-2\\ \end{array}\right) + \left(\begin{array}{c} n\\ d-4\\ \end{array}\right),
$$
which is the coefficient of $t^d$ in
$\displaystyle \frac{(1-t^2)^{n+2}}{(1-t)^n}$, the dimension predicted by the Fr\"oberg conjecture for the quotient of the polynomial ring in $n$ variables by $n+2$ general forms of degree 2. Since the dimension is always greater than or equal to the dimension predicted by the Fr\"oberg conjecture, it follows that equality must hold. Observe that $\F[x_1, \ldots, x_n]/(x_1^2, \ldots, x_n^2, \ell_1^2, \ell _2^2))$ is obtained as the quotient of $\F[x_1, \ldots, x_{n+1}]/(x_1^2, \ldots, x_{n+1}^2, (x_1+\cdots + x_{n+1})^2)$ by a general linear form. The conclusion of Theorem (\ref{injectivity for squares}) now follows by applying this fact with $n-1$ in the role of $n$ and using Observation (\ref{Froberg}). In Section 5, we will prove that injectivity fails in degree $d$ for $n\leq 3d-3$, and we will discuss how the failure of injectivity relates to Conjecture 2 from \cite{CriQ-Lun-Nen}, which asserts that

\begin{equation}\label{exterior conj}
\mathrm{dim}_{\F}\big((P_n/(x_1^2, \ldots, x_n^2, \ell_1^2, \ell _2^2))_d\big)=a_{n, d}
\end{equation}
holds for all $n, d$.

The following corollary shows that Fr\"oberg's Conjecture holds for $n+2$ quadratic linear forms up to degree $D:= \lfloor \frac{n+2}{3} \rfloor$.

\begin{corollary}
Let $R=P_{n}/(x_{1}^{2},\dots,x_{n}^{2},\ell_{1}^{2},\ell_{2}^{2})$ where $\ell_{1}$ and $\ell_{2}$ are general linear forms. We have
\begin{equation*}
    \big[\mathrm{HS}(R,t)\big]_{D}=\left[\frac{(1-t^{2})^{n+2}}{(1-t)^{n}}\right]_{D}.
\end{equation*}
where $D:=\lfloor\frac{n+2}{3}\rfloor$
\end{corollary}

\begin{proof}
This follows immediately from the preceding theorem and Observation (\ref{Froberg}).
\end{proof}

\section{Injectivity for ideals generated by cubes}

\begin{notation}
Let $n\geq 1$ be an integer. We fix the following:
\begin{enumerate}[itemsep=5pt]
    \item $P_n$ denotes the polynomial ring $\F[x_1,\dots,x_n]$ where $\F$ is a field of characteristic $0$. In addition, the notation $P_{n,d}$ will refer to the $d^{\text{th}}$ graded piece of $P_n$.
    \item The ideal $J_n:=\big(x_1^3,\dots,x_{n}^3, (x_1+\cdots+x_n)^3\big)$ in the ring $P_n$.\newline
    Further, $\mathrm{in}(J_n)$ is the initial ideal of $J_n$ with respect to the reverse lexicographic order.
    \item For $k\ge 3$, $D_k$ is the set of monomials of degree $k$  in the variables $x_1, \ldots, x_{k-1}$ which are not divisible by $x_{k-1}^2$. (If $n<k-1$, then $D_k$ contains all monomials of degree $k$ in $x_1, \ldots, x_n$). For instance, the first few $D_k$ are given by
    \begin{equation*}
        D_3=\{x_1^2x_2\},\quad D_4=\{x_1^2x_2^2,\, x_1^2x_2x_3,\, x_1x_2^2x_3\},\quad\text{and}
    \end{equation*}
    \begin{equation*}
        D_5=\{x_1^2x_2^2x_3,\, x_1^2x_2^2x_4,\, x_1^2x_2x_3^2,\, x_1^2x_2x_3x_4,\, x_1x_2^2x_3^2,\, x_1x_2^2x_3x_4,\, x_1x_2x_3^2x_4,\, x_2^2x_3^2x_4, x_1^2x_3^2x_4\}.
    \end{equation*}
    \item The ideal $K_n\subseteq P_n$ generated by $x_1^3, \ldots, x_n^3$, and $\displaystyle \bigcup_{k=3}^{n+1} D_k$.
    \item We define $
    G(n, d):=\displaystyle\mathrm{dim}\left(\left(\frac{P_n}{(x_1^3, \ldots, x_n^3)}\right)_d\right)$.
\end{enumerate}
\end{notation}

We make two preliminary observations regarding the notation above.

\begin{obs}\label{small variable number}
For $n\leq d-2$, we have $(K_n)_d=(P_n)_d$.
\end{obs}

\begin{proof}
Let $m\in P_{n}$ be a monomial of degree $d$. If $x_{n}^2$ does not divide $m$, then $m$ lies in $D_{n+1}$ by definition and we are done. If instead $x_{n}^2$ divides $m$, then we can write $m=x_{n}^2m'$ where the degree of $m'$ is $d-2$. If $x_{n}$ divides $m'$, then $x_{n}^3$ divides $m$ and so $m$ lies in $K_{n}$. If not, then write $m=x_{n}(x_{n}m')$. The monomial $x_{n}m'$ is not divisible by $x_{n}^2$, so $x_{n}m'$ lies in $D_{n}$. This implies that $m$ belongs to $(K_{n})_d$. 
\end{proof}

\begin{obs}\label{cube-rec}
The function $G(n,d)$ satisfies the following two relations:
\begin{enumerate}[label=(\alph*), itemsep=5pt]\centering
    \item $G(n, d) = G(n-1, d) + G(n-1, d-1)+G(n-1, d-2)$;
    \item $G(d, d) - G(d, d-1) = G(d-1, d-2)-G(d-1, d-3)$.
\end{enumerate}
\end{obs}

\begin{proof}
For (a), note that $G(n, d)$ is the number of monomials $x_1^{i_1}\cdots x_n^{i_n}$ with $\sum_{j=1}^ni_j=d$ and $0\le i_j \le 2$ for all $j\in \{1, \ldots, n\}$. The set of such monomials can be partitioned into monomials with $i_n=0, i_n=1$, and $i_n=2$, which correspond to $G(n-1, d), G(n-1, d-1)$, and $G(n-1, d-2)$, respectively. 

Now, (b) follows from (a) and the fact that $G(d-1, d)=G(d-1, d-2)$. The latter fact is because when $n$ is $d-1$, the socle degree of $\displaystyle \frac{P_{d-1}}{(x_1^3, \ldots, x_{d-1}^3)}$ is equal to $2d-2$, and the Hilbert function is symmetric about half of the socle degree by Theorem (\ref{CI H-func & soc deg}).
\end{proof}

The main result of this section is Theorem (\ref{inj for cubes}), which shows that if $\displaystyle n\ge \frac{3d-3}{2}$, the map given by multiplication by a general linear form from $(P_n/J_n)_{d-1}$ to $(P_n/J_n)_d$ is injective. 

The key ingredient is the following description of the initial ideal of $J_n$.

\begin{theorem}\label{initial3}
Using the reverse lexicographic order, we have $\mathrm{in}(J_n)=K_n$.
\end{theorem}

\begin{proof}
For each degree $d\ge 3$, we want to prove that $(\mathrm{in}(J_n))_d=(K_n)_d$. The statement is clear for $d=3$. 

For all $n, d$, we will prove
\begin{enumerate}[label=(\roman*)]
\item \begin{equation}\label{eq}
\mathrm{dim}\left(\frac{P_{n,d}}{(K_n)_d}\right) = \mathrm{dim}\left(\frac{P_{n,d}}{(J_n)_d}\right)
\end{equation}

\item $(K_n)_d\subseteq (\mathrm{in}(J_n))_d$.
\end{enumerate}

We prove both parts by induction on $d$. The statement is clear for $d=3$.

We calculate the right-hand side of (\ref{eq}). Note that 
$\displaystyle \frac{\F[x_1, \ldots, x_n]}{J_n}$ is isomorphic to the cokernel of the map $\displaystyle \times L : \frac{\F[x_1, \ldots, x_{n+1}]}{(x_1^3, \ldots, x_{n+1}^3)} (-1) \to \frac{\F[x_1, \ldots, x_{n+1}]}{(x_1^3, \ldots, x_{n+1}^3)}$, where $L=\sum_{i=1}^{n+1} x_i$. Since the WLP holds for the complete intersection, Remark (\ref{HF of ACI}) yields 

\begin{align}\label{wlp relation}
    \mathrm{dim}\left(\frac{P_{n,d}}{(J_n)_d}\right)&=\mathrm{max}\{G(n+1, d)-G(n+1, d-1), 0\}\\
    &=\nonumber\begin{cases}
        G(n+1, d)-G(n+1, d-1) & \mathrm{if}\ d\le n+1,\\
        0 & \text{otherwise}
    \end{cases}
\end{align}

\noindent where the last equality follows from Theorem (\ref{CI H-func & soc deg}).

For $n\le d-2$, (\ref{wlp relation}) shows that $(J_n)_d$ consists of all homogeneous polynomials of degree $d$ in $P_n$; the same holds for $(K_n)_d$ by Observation (\ref{small variable number}), and therefore we have $(\mathrm{in}(J_n))_d=(J_n)_d=(K_n)_d$, as desired. For the rest of the proof, we will assume $n\ge d-1$.

For $n=d-1$, the left-hand side of (\ref{eq}) is equal to the number of monomials of the form $x_{d-1}^2m$, where $m$ is a monomial in the variables $x_1, \ldots, x_{d-2}$ with $m\notin K_{d-2}$ and $\mathrm{deg}(m)=d-2$; therefore it is equal to $\displaystyle \mathrm{dim}\left(\frac{\F[x_1, \ldots, x_{d-2}]_{d-2}}{(K_{d-2})_{d-2}}\right)$. By the inductive hypothesis and (\ref{wlp relation}), this is equal to
\begin{equation*}
    \mathrm{dim}\left(\frac{\F[x_1, \ldots, x_{d-2}]_{d-2}}{(J_{d-2})_{d-2}}\right)=G(d-1, d-2)-G(d-1, d-3).
\end{equation*}
The right-hand side of (\ref{eq}) is equal to $G(d, d)-G(d, d-1)$ by (\ref{wlp relation}). The equality asserted in (\ref{eq}) now follows from Observation (\ref{cube-rec}) part (b).

For $n\ge d$, let $E(n,d)$ and $F(n,d)$ denote the left-hand side and right-hand side respectively of equation (\ref{eq}). Using the recursive equation from Observation (\ref{cube-rec}) part (a), we obtain
\begin{align*}
    F(n,d)&=(G(n, d)-G(n, d-1))+(G(n, d-1)-G(n, d-2))+(G(n, d-2)-G(n,d-3))\\
    &=F(n-1, d)+F(n-1,d-1)+F(n-1, d-2),
\end{align*}
where the last equality uses equation (\ref{wlp relation}) and the assumption $n\ge d$.

We claim that the same recurrence is satisfied by $E(n, d)$. Indeed, $E(n, d)$ is the number of monomials of degree $d$ of the form $x_1^{a_1}\cdots x_n^{a_n}$ with $a_1, \ldots, a_n\in \{0, 1, 2\}$ that are not divisible by any monomial in $D_k$ for any $k\le d$. These monomials can be partitioned according to the value of $a_n \in \{0, 1, 2\}$. The number of such monomials with $a_n=0$ is $E(n-1, d)$, the number of such monomials with $a_n=1$ is $E(n-1, d-1)$, and the number of such monomials with $a_n=2$ is $E(n-1, d-2)$. (Note that the assumption $n\ge d$ ensures that the set $D_k$ in $P_n$ is the same as the set $D_k$ in $P_{n-1}$ for all $k\le d$.)

The conclusion now follows by induction on $n$. This ends the proof of (i).\vspace{12pt}

Proof of (ii): We observe that it is sufficient to prove that the inclusion $(K_n)_d\subseteq (\mathrm{in}(J_n))_d$ holds for $n=d-1$. Recall that we already proved equality holds for $n\le d-2$. Assuming that the inclusion holds for $n=d-1$, consider $n\ge d$ and induct on $n$. Let $\overline{\ \ \  }$ denote modding out by $x_n$. We have $\overline{(K_n)_d}=(K_{n-1})_d\subseteq (\mathrm{in}(J_{n-1}))_d=\overline{(\mathrm{in}(J_n))_d}$. The first equality is clear from the definition of $K_n$, the middle inclusion is the inductive hypothesis, and the second equality is from properties of the initial ideal with respect to the reverse lexicographic order. Since the generators of $K_n$ in degrees $\le d$ only involve the variables $x_1, \ldots, x_{n-1}$, it follows that $(K_n)_d\subseteq (\mathrm{in}(J_n))_d$.

Now let $n=d-1$. We wish to prove the inclusion $(K_{d-1})_d\subseteq (\mathrm{in}(J_{d-1}))_d$, but in light of (i) this is equivalent to proving the opposite inclusion $(K_{d-1})_d\supseteq (\mathrm{in}(J_{d-1}))_d$. Aiming to prove the latter, we will induct on $d$ with the statement being obvious for $d=3$.

Note that $(K_{d-1})_d$ consists of all monomials of degree $d$ in the variables $x_1, \ldots, x_{d-2}$, as well as all monomials of the form $x_{d-1}m$ where $m$ is a monomial of degree $d-1$ in the variables $x_1, \ldots, x_{d-2}$ (all of these belong to $D_d$); additionally, a monomial of the form $x_{d-1}^2m$ (with $m$ a monomial of degree $d-2$ in the variables $x_1, \ldots, x_{d-2}$) is in $K_{d-1}$ if and only if $m\in K_{d-2}$.

Therefore, the inclusion $(\mathrm{in}(J_{d-1}))_d\subseteq (K_{d-1})_d$ is equivalent to the condition that the only degree $d$ multiples of $x_{d-1}^2$ that are in $\mathrm{in}(J_{d-1})$ are the multiples of $x_{d-1}^3$ and those obtained as $x_{d-1}^2m$ with $m\in (K_{d-2})_{d-2}$. In other words, we want to prove that 
\begin{equation*}
((x_{d-1}^2) \cap \mathrm{in}(J_{d-1}))_d=(x_{d-1}^2K_{d-2}+(x_{d-1})^3)_d
\end{equation*}
or, equivalently,
\begin{equation}\label{desired}
(\mathrm{in}(J_{d-1}): x_{d-1}^2)_{d-2}=(K_{d-2}+(x_{d-1}))_{d-2}.
\end{equation}
In order to calculate the left-hand side of (\ref{desired}), note that 
$\mathrm{in}(J_{d-1}):x_{d-1}^2=\mathrm{in}(J_{d-1}:x_{d-1}^2)$ by Fact (\ref{revlex props}).
We have 
\begin{equation}\label{col}
J_{d-1}:x_{d-1}^2= (x_{d-1}) + (x_1^3, \ldots, x_{d-2}^3, (x_1+\cdots +x_{d-1})^3):(x_{d-1}^2)
\end{equation}
Using the change of variables $x_{d-1}':=x_1+\cdots +x_{d-1}$, note that $(x_1^3, \ldots, x_{d-2}^3, (x'_{d-1})^3)$ is a complete intersection ideal and therefore $\F[x_1, \ldots, x_{d-2}, x_{d-1}']/(x_1^3, \ldots, (x'_{d-1})^3)$ has the Strong Lefschetz Property.
 Thus, the map given by multiplication by $x_{d-1}^2=(x'_{d-1} -x_1-\cdots - x_{d-2})^2$
\begin{equation}\label{map}
{\times}\, x_{d-1}^2:\left( \frac{\F[x_1, \ldots, x_{d-2}, x_{d-1}']}{(x_1^3, \ldots, (x'_{d-1})^3)}\right)_{d-2}\ \longrightarrow \ \left(\frac{\F[x_1, \ldots, x_{d-2}, x_{d-1}']}{(x_1^3, \ldots, x_{d-1}'^3)}\right)_d
\end{equation}
has maximal rank. The socle degree of the complete intersection is $2d-2$. Since the Hilbert function of the complete intersection is symmetric about $d-1$ (half of the socle degree) by Theorem (\ref{CI H-func & soc deg}), it follows that the domain and target of the map in (\ref{map}) have the same dimension; therefore, the map is bijective and so, in particular, injective. This implies
$$
((x_1^3, \ldots, x_{d-2}^3, (x'_{d-1})^3):x_{d-1}^2)_{d-2} = (x_1^3, \ldots, x_{d-2}^3, (x'_{d-1})^3)_{d-2}.
$$
Equation (\ref{col}) now implies
\begin{equation*}
(J_{d-1}:(x_{d-1}^2))_{d-2}=((x_{d-1}) + (x_1^3, \ldots, x_{d-2}^3, (x_1+\cdots +x_{d-1})^3))_{d-2}.
\end{equation*}
Hence, the left-hand side of (\ref{desired}) is equal to the degree $d-2$ piece of the initial ideal of $(x_{d-1}) + (x_1^3, \ldots, x_{d-2}^3, (x_1+\cdots +x_{d-1})^3)$, which is equal to $((x_{d-1})+\mathrm{in}(J_{d-2}))_{d-2}$. By the inductive hypothesis, this is equal to $((x_{d-1})+K_{d-2})_{d-2}$ as desired.
\end{proof}

Before proving the main result of this section, we need a few preliminary results to set the stage for an appeal to Lemma (\ref{ayden}).

The following observation is an immediate consequence of the definition of $K_n$.

\begin{obs}\label{notin}
Let $\mu \in P_n$ be a monomial of degree $d$. Assume $\mu \notin K_n$.

a. If $j\ge d+1$ and $x_j^2$ does not divide $\mu$, then $\mu x_j\notin K_n$.

b. If $\mu $ is divisible by $x_d$ but not by $x_d^2$, then $\mu x_d\notin K_n$.
\end{obs}

In the next two results, we will use the following terminology: we say that a monomial $m$ ends in $x_j^{a_j}$ if $m=x_1^{a_1}\cdots x_{j-1}^{a_{j-1}}x_j^{a_j}$. The lemmas below are roughly analogous to Lemmas (\ref{degree}) and (\ref{degree2}), when the defining ideal was generated by squares.

\begin{lemma}\label{deg3}
Let $\mu \in P_d$ be a monomial in the variables $x_1, \ldots, x_d$.

a. If $\mu \notin K_d$ and $\mathrm{deg}(\mu )=\delta\le d-1$ then there exists a monomial $\widetilde{\mu}\in P_d$ with the following three properties: 

(i) $\widetilde{\mu}$ is a multiple of $\mu$ of degree $d$, 
(ii) $\widetilde{\mu}\notin K_d$, and (iii) $\widetilde{\mu}$ ends in $x_d$ or $x_d^2$.

b. If $\mu \notin K_d$, $\mathrm{deg}(\mu )=d$, and $\mu $ does not end in $x_{d-1}^2$, then there exists $\widetilde{\mu}\in P_d$ such that $\widetilde{\mu}$ is a multiple of $\mu$ of degree $d+1$ and $\widetilde{\mu}\notin K_d$.
\end{lemma}

\begin{proof}
a. We prove this by using induction on $d$. The statement is obvious for $d=1$. Let $d\ge 2$. Note that it suffices to construct a monomial $\widetilde{\mu}$ of degree $\delta +1$, and then iterate the construction.

If $\mu$ is not divisible by $x_d^2$, let $\widetilde{\mu}:=\mu x_d$. We have $\widetilde{\mu} \notin K_d$ by Observation (\ref{notin}) (a).  Otherwise, say $\mu = x_d^2\mu '$, where $\mu '\in P_{d-1}$, $\mathrm{deg}(\mu ')=\delta -2 <d-2$. $\mu \notin K_d$ implies $\mu'\notin K_{d-1}$.  By the inductive hypothesis, there is a $\widetilde{\mu '}\in P_{d-1}$ a multiple of $\mu '$ of degree $\delta -1$ with $\widetilde{\mu '} \notin K_{d-1}$. Then $\widetilde{\mu}=\widetilde{\mu'} x_d^2$ satisfies the desired conditions.

b. The assumption that $\mu $ does not end in $x_{d-1}^2$ means that  $x_d$ must divide $\mu$ (otherwise, $\mu $ would be a monomial of degree $d$ in the variables $x_1, \ldots, x_{d-1}$, not divisible by $x_{d-1}^2$, thus $\mu \in K_d$).

If $\mu$ is not divisible by $x_d^2$, let $\widetilde{\mu}:=\mu x_d$. We have $\widetilde{\mu}\notin K_d$ by Observation (\ref{notin}) (b).
Otherwise, say $\mu = \mu ' x_d^2$ with $\mu' \in P_{d-1}$ and $\mathrm{deg}(\mu')=d-2$. From part (a), there exists $\widetilde{\mu'}\in P_{d-1}$ a multiple of $\mu '$ of degree $d-1$ with $\widetilde{\mu '}\notin K_{d-1}$. It follows that $\widetilde{\mu }:= \widetilde{\mu'} x_d^2 \notin K_d$.
\end{proof}

In the next result, we use the following notation:

\begin{notation}
Let $K_{n, d}\subseteq P_n$ denote the ideal generated by those generators of $K_n$ that have degree $\le d$ (note that these generators only involve the variables $x_1, \ldots, x_{d-1}$). Note that for a monomial $\mu$ of degree $d$, $\mu \in K_n\Leftrightarrow \mu \in K_{n, d}$.
\end{notation}

\begin{lemma}\label{deg4}
Let $\mu \in P_n$ be a monomial of degree $d-1$ which does not end in $x_{d-2}^2$. Assume $\mu \notin K_n$. Then there exists a monomial $M$ which is a multiple of $\mu $ of degree $2n-d+2$ such that $M \notin K_{n, d}$.
\end{lemma}

\begin{proof}
If $\mu$ only involves the variables $x_1, \ldots, x_{d-1}$, apply Lemma (\ref{deg3})(b) (with $d-1$ in the role of $d$) to obtain $\widetilde{\mu}$ of degree $d$ in the variables $x_1, \ldots, x_{d-1}$, $\widetilde{\mu }\notin K_n$. Then $M:=\widetilde{\mu } x_d^2\cdots x_n^2$ has the desired properties.

Otherwise, say $\mu =\mu ' x_{j_1}^{\alpha_1}\cdots x_{j_s}^{\alpha_s}$ with $j_1, \ldots, j_s\in \{d, \ldots, n\}$, $\alpha_1, \ldots, \alpha_s\in \{1, 2\}$, and $\mu ' \in P_{d-1}$ has degree $\le d-2$. Apply Lemma (\ref{deg3})  to $\mu '$ in the role of $\mu $ and $d-1$ in the role of $d$. It follows that there is a $\widetilde{\mu '}\in P_{d-1}$ of degree $d$ which is a multiple of $\mu '$ and $\widetilde{\mu '}\notin K_{d-1}$. Now $M:=\widetilde{\mu'}x_d^2\cdots x_n^2$ has the desired properties.
\end{proof}

\begin{remark}\normalfont{
The assumption that $\mu$ does not end in $x_{d-2}^2$ is crucial for Lemmas (\ref{deg3}) and (\ref{deg4}). Indeed, let $\mu\in P_n$ have $\deg{(\mu)}=d-1$, $\mu$ ends in $x_{d-2}^2$, and $\mu\notin K_n$. We claim that if a monomial $M$ is a multiple of $\mu$ with degree $2n-d+2$, then $M\in K_{n,d}$. As a multiple of $\mu$, we may write
\begin{equation*}
    M=\mu\cdot x_{1}^{\alpha_1}\cdots x_{n}^{\alpha_n}\quad\text{where}\quad \sum_{j=1}^{n}\alpha_j=2n-2d+3.
\end{equation*}
If $\alpha_j\geq 3$ for any $j$, then $M$ is a multiple of a cube and therefore $M\in K_{n,d}$ provided $d\geq 3$. Now suppose that $\alpha_j\in\{0,1,2\}$ for all $j$. There are two cases:
\begin{enumerate}
    \item If $\alpha_j\neq 0$ for some $j\in\{1,\dots,d-1\}$, then consider $\widetilde{\mu}=\mu\cdot x_j$. As $\mu$ ends in $x_{d-2}^2$, it follows that $\widetilde{\mu}$ is not divisble by $x_{d-1}^2$, so $\widetilde{\mu}\in K_{n,d}$. Moreover, $\widetilde{\mu}$ is a divisor of $M$ and so $M\in K_{n,d}$ in this case.
    \item If $\alpha_j=0$ for all $j\in\{1,\dots,d-1\}$, then $M$ has the form $M=\mu\cdot x_{d}^{\alpha_d}\cdots x_{n}^{\alpha_n}$. The maximum possible degree of $M$ is given by choosing $\alpha_j=2$ for all $d\leq j\leq n$. Thus,
    \begin{equation*}
        \deg{(M)}\leq (d-1)+2(n-d+1)=2n-d+1.
    \end{equation*}
    This contradicts the assumption that $\deg{(M)}=2n-d+2$.
\end{enumerate}
Because of this, the case when $\mu$ ends in $x_{d-2}^2$ will need to be treated separately in the proof of our theorem below.}
\end{remark}

\begin{theorem}\label{inj for cubes}
If $n\ge \frac{3d-3}{2}$, then the multiplication by a general linear form from $(P_n/J_n)_{d-1}$ to $(P_n/J_n)_d$ is injective.
\end{theorem}

\begin{proof}
By Proposition (\ref{wiebe}), it suffices to prove injectivity when $J_n$ is replaced by $\mathrm{in}(J_n)=K_n$.
Also note that the injectivity of the map in question is not affected by the generators of $K_n$ of degree $>d$, so we may replace $K_n$ by $K_{n, d}$

Assume $Q=\sum a_{\mu } \mu $ is in the kernel of the map given by multiplication by $L=\sum_{i=1}^n x_i$ from $(P_n/K_n)_{d-1}$ to $(P_n/K_n)_d$, where the summation is over all monomials $\mu $ of degree $d-1$. We need to prove that $a_{\mu }=0$ for all $\mu \notin K_n$.

For each monomial $\mu $ of degree $d-1$ that does not end in $x_{d-2}^2$, Lemma (\ref{deg4}) shows that we can construct a monomial $M$ of degree $2d-1$ which is a multiple of $\mu$ and is not in $K_{n, d}$, and Lemma (\ref{ayden}) implies $a_{\mu }=0$.

Consider $\mu $ of degree $d-1$ which ends in $x_{d-2}^2$. Then $\mu x_d\notin K_n$. The assumption $LQ=0$ implies that the coefficient of $\mu x_d$ in $LQ$ must be zero, and therefore 
$\sum_{\nu } a_{\nu}=0$, where the summation is over all divisors $\nu $ of $\mu x_d$ of degree $d-1$. One of these divisors is $\mu $; the others are $\mu ' x_d$ where $\mu '$ is a divisor of $\mu $ of degree $d-2$. Lemma (\ref{deg4}) implies that $a_{\mu ' x_d}=0$ for all such $\mu ' x_d$, and therefore $a_{\mu }$ is also zero.
\end{proof}

The following shows that the Fr\"oberg conjecture holds for $n+2$ linear forms of degree 3 up to $D:=\lfloor\frac{2n}{3}\rfloor+1$:
\begin{corollary}
Let $R=P_{n}/(x_{1}^{3},\dots,x_{n}^{3},\ell_{1}^{3},\ell_{2}^{3})$ where $\ell_{1}$ and $\ell_{2}$ are general linear forms. We have
\begin{equation*}
    \big[\mathrm{HS}(R,t)\big]_{D}=\left[\frac{(1-t^{3})^{n+2}}{(1-t)^{n}}\right]_{D}.
\end{equation*}
where $D:=\lfloor\frac{2n}{3}\rfloor+1$.
\end{corollary}

\begin{proof}
This follows at once from the preceding theorem and Observation (\ref{Froberg}).
\end{proof}

\begin{remark}
Although we do not prove it in general, experimental evidence from \texttt{Macaulay2} indicates that the bound in Theorem (\ref{inj for cubes}) is sharp when the number of variables $n\le 10$.
\end{remark}

\section{Failure of injectivity for ideals generated by squares}

Recall the ideal $I_n:=\big(x_1^2,\dots,x_{n}^2, (x_1+\cdots+x_n)^2\big)$ in the ring $P_n$. Further, $\mathrm{in}(I_n)$ is the initial ideal of $I_n$ with respect to the reverse lexicographic order. The main result of this section is the following:

\begin{theorem}\label{square injectivity failure}
Let $d>2$ and $n<3d-2$. For $\ell$ a general linear form, the multiplication by $\ell $ map from $(P_n/I_n)_{d-1}$ to $(P_n/I_n)_d$ fails to be injective.
\end{theorem}
Note that when $n=3d-3$ and $d\ge 3$, we have $\dim(P_n/I_n)_d \ge \dim(P_n/I_n)_{d-1}$. Hence, the failure of injectivity for the multiplication by a general linear form is indeed equivalent to the failure of WLP at degree $d$.

\begin{remark}\normalfont{
The first case where Theorem (\ref{square injectivity failure}) applies is $d=3, n=6$. For this case, the failure of injectivity also follows from a result in \cite{Froberg-Hollman}, where it is shown that an ideal generated by the squares of 7 generic forms of degree 2 in 5 variables fails to have the Hilbert function predicted by Fr\"{o}berg (equation (\ref{froberg})) in degree 3.}
\end{remark}
The proof of Theorem~(\ref{square injectivity failure}) is constructive; an explicit nonzero element in the kernel of multiplication by $\ell$ will be provided. In preparation for this, we define two recursive functions. For $0\le t \le d-1$, define $\varepsilon(t)$ as follows: $\varepsilon(0)=1$, and for $t\ge 1$, $\displaystyle \varepsilon(t):=-\frac{t\varepsilon(t-1)}{d-t}$. For any two subsets $\mathcal{I}$ and $\mathcal{J}$ of $\{1\, \ldots, n\}$, let $\varepsilon_\mathcal{I}^\mathcal{J}:=\varepsilon(|\mathcal{I} \cap \mathcal{J}|)$. Further, we define $\psi (t)$ by $\psi (0)=1, \psi (1)=-2/(d-2)$, and for $2\le t\le d-2$, $\psi (t)$ satisfies 
$$
\left(\begin{array}{c} d-t\\ 2 \\ \end{array}\right) \psi (t) +t(d-t)\psi (t-1) + \left( \begin{array}{c} t\\ 2 \\ \end{array}\right) \psi (t-2)=0.
$$
For any two subsets $\I$ and $\J$ of $\{1, \ldots, n\}$, define the notation $\psi _\I^\J := \psi (|\I \cap \J|)$.

These two functions have the following properties that will prove useful later.

\begin{obs}\label{epsilon and psi sums}
Let $d>2$ and $n<3d-2$. Then the functions $\varepsilon(t)$ and $\psi(t)$ satisfy the following:
\begin{enumerate}[label=(\alph*)]
    \item if $\I$ and $\J$ are subsets of $\{1,\ldots, n\}$ with $|\I|=d$ and $|\J|=n-2d+2$, we have
    \begin{equation*}
        \sum_{i\in\I}  \varepsilon_{\I\setminus\{i\}}^{\J} =\left\lbrace \begin{array}{cll} d & \mathrm{if}\ \I \cap \J=\emptyset, \\ 0 & \mathrm{otherwise}. \end{array}\right.
    \end{equation*}
    \item if $\I$ and $\ml$ are subsets of $\{1, \ldots, n\}$ with $|\I|=d$ and $|\ml|=d-2$, we have 
    $$\sum_{\{u, v\}\subseteq \I}\psi_{\I\, \backslash \, \{u, v\}}^\ml =\left\lbrace \begin{array}{cll} \left(\begin{array}{c} d\\ 2 \\ \end{array}\right) & \mathrm{if}\ \I \cap \ml =\emptyset, \\ 0 & \mathrm{otherwise}. \end{array}\right.$$
\end{enumerate}
\end{obs}

\begin{proof}
(a) Note that we have $n-2d+2\le d-1$, so $|\I \cap \J|$ is also bounded above by $d-1$. In the case when $\I\cap \J=\emptyset $, each of the $d$ terms in the summation is equal to $\varepsilon(0)=1$, while in the case when $|\I\cap \J|=t>0$, there are $t$ terms that are equal to $\varepsilon(t-1)$ (when $i\in \I\cap \J$), and $d-t$ terms that are equal to $\varepsilon(t)$ (when $i\notin \I\cap \J$). Therefore, the summation is equal to $t\varepsilon(t-1)+(d-t)\varepsilon(t)=0$.

(b) If $\I \cap \ml=\emptyset$, then each term in the summation is equal to $\psi (0)=1$. If $|\I \cap \ml|=1$, then $d-1$ of the terms in the summation are equal to $\psi (0)$ (when $u$ or $v$ is in the intersection) and the remaining $\displaystyle \left( \begin{array}{c} d-1 \\ 2 \\ \end{array}\right)$ terms are equal to $\psi (1)$, whence the sum is equal to $(d-1)\psi (0) +  \left( \begin{array}{c} d-1 \\ 2 \\ \end{array}\right)\psi (1) =0$. Lastly, if $|\I \cap \ml|=t\ge 2$, then $\displaystyle \left(\begin{array}{c} t \\ 2\\ \end{array}\right) $ of the terms in the summation are equal to $\psi (t-2)$ (when both $u$ and $v$ are in the intersection), $t(d-t)$ of the terms are equal to $\psi (t-1)$ (when exactly one of $u$ or $v$ is in the intersection), and the remaining $\displaystyle \left(\begin{array}{c} d-t \\ 2 \\ \end{array}\right) $ terms are equal to $\psi (t)$ (when neither $u$ nor $v$ are in the intersection); therefore, the summation is equal to 0 by the definition of $\psi (t)$.
\end{proof}

We proceed to the proof of Theorem (\ref{square injectivity failure}).

\begin{proof}[Proof of (\ref{square injectivity failure})]

For any subset $\I\subseteq \{1, \ldots, n\}$, we let $a_{\I}:=\prod_{i\in \I} a_i$ and $x_{\I}:=\prod_{i\in \I}x_i$. Define a form of degree $d-1$ as follows:
$$
Q:=\sum_{|\I|=d-1} a_\I\left(\sum_{|\J|=n-2d+2}  \varepsilon_{\I}^{\J}a_\J\right)x_\I
$$
where the summations are over all subsets $\I, \J\subseteq \{1, \ldots, n\}$ with the indicated cardinalities.

We verify that $LQ\in J_n$. 

For each $\I \subseteq\{1, \ldots, n\}$ with $|\I|=d$, the coefficient of $x_{\I}$ in $LQ$ is
$\displaystyle 
\sum_{i\in \I}a_i b_{\I \, \backslash \{i\}},
$
where $$ b_{\I \, \backslash \, \{i\}}=a_{\I \, \backslash \, \{i\}} \cdot \sum_{|\J|=n-2d+1} \varepsilon_{\I \, \backslash \, \{i\}}^\J a_\J$$ is the coefficient of $x_{\I \, \backslash \, \{i\}}$ in $Q$.
Since $a_ia_{\I \, \backslash \, \{i\}}=a_{\I}$, we obtain 
\begin{equation}\label{cff}
a_{\I}\left( \sum_{i\in\I} \sum_{|\J|=n-2d+1} \varepsilon_{\I \, \backslash \, \{i\}}^\J a_\J \right) = a_{\I}\sum_{|\J|=n-2d+2}\left(\sum_{i\in \I}  \varepsilon_{\I \, \backslash \, \{i\}}^\J \right) a_\J = d\cdot a_{\I} \sum_{\substack{\J\cap \I=\emptyset \\ |\J|=n-2d+2}}a_\J
\end{equation}
where the second equality follows from part (a) of Observation (\ref{epsilon and psi sums}).

We will now define a polynomial $Q'$ of degree $d-2$ such that \begin{equation}\label{congr}LQ\equiv Q'(x_1+\cdots +x_n)^2\ \ \ \ \mathrm{mod} \ (x_1^2, \ldots, x_n^2).\end{equation} This will prove $LQ\in J_n$.

We define $Q'$ as follows:
$$
Q'=\sum_{|\K|=d-2} \left(\sum_{|\ml|=d-2} \psi_\K^\ml \, \frac{a_1\cdots a_n}{a_\ml}\right)  x_\K
$$
where the summations are over all subsets $\K, \ml\subseteq \{1, \ldots, n\}$ with the indicated cardinality.
Let $\displaystyle c_\K:=\sum_{|\ml|=d-2} \psi_\K^\ml \, \frac{a_1\cdots a_n}{a_\ml}$ denote the coefficient of $x_\K$ in $Q'$.

In order to verify equation (\ref{congr}), we calculate the coefficient of $x_\I$ in $Q'(x_1+\cdots +x_n)^2$ for every subset $\I\subseteq \{1, \ldots, n\}$ with $|\I|=d$, and we confirm that it is equal to the corresponding coefficient of $LQ$ (calculated in (\ref{cff})).
The coefficient in question is equal to 
$$
\sum_{\{u, v\}\subseteq \I} c_{\I\, \backslash \{u, v\}}=
\sum_{\{u, v\} \subseteq \I}\,  \sum_{|\ml|=d-2} \psi_{\I\, \backslash \, \{u, v\}}^\ml\frac{a_1\cdots a_n}{a_\ml} = \sum_{|\ml|=d-2} \left(\sum_{\{u, v\} \subseteq \I}\psi_{\I\, \backslash \, \{u, v\}}^\ml\right) \frac{a_1\cdots a_n}{a_\ml}.
$$

\noindent By part (b) of Observation (\ref{epsilon and psi sums}), the coefficient of $x_\I$ in $Q'(x_1+\cdots +x_n)^2$ is equal to 
\begin{equation}\label{cff2}
\binom{d}{2}\sum_{\substack{L\cap I=\emptyset \\ |\ml|=d-2}} \frac{a_1\cdots a_n}{a_\ml}.
\end{equation}
For each term in (\ref{cff2}), $\displaystyle \frac{a_1\cdots a_n}{a_\ml}=a_\I a_\J$ where $\J\subseteq \{1, \ldots, n\}$ is disjoint from $\I$ and satisfies $|\J|=n-2d+2$, whence the right-hand side of (\ref{cff2}) is the same as the right-hand side of (\ref{cff}).

It remains to be shown that $Q\notin I_n$ for general choices of $a_1, \ldots, a_n$.

Assume $Q\in I_n$.  Let $b_{i_1\cdots i_{d-1}}$ denote the coefficient of $x_{i_1}\cdots x_{i_{d-1}}$ for each $\{i_1, \ldots, i_{d-1}\}\subseteq \{1, \ldots, n\}$. The assumption that $Q\in I_n$ implies that the coefficients $b_{i_1\cdots i_{d-1}}$ satisfy certain linear equations.

Indeed, if $\displaystyle C=\sum_{j_1<\cdots < j_{d-3}}c_{j_1\cdots j_{d-3}}x_{j_1}\cdots x_{j_{d-3}}$ is such that 
$Q=C(x_1+\cdots + x_n)^2$ (mod $(x_1^2, \ldots, x_n^2)$), then 
\begin{equation}\label{rel}
b_{i_1\cdots i_{d-1}}=\sum_{u<v} c_{i_1\cdots \hat{i_u}\cdots\hat{i_v}\cdots i_{d-1}}
\end{equation}
Let $\mathcal{C}$ denote the matrix with rows indexed by subsets $\{i_1, \ldots, i_{d-1}\}\subseteq \{1, \ldots, n\}$, columns indexed by subsets $\{j_1, \ldots, j_{d-3}\}\subseteq\{1, \ldots, n\}$, and entry equal to 1 if the column index is a subset of the row index, 0 otherwise. Arranging  the coefficients $b_{i_1\cdots i_{d-1}}$ in a vector ${\bf b}$ with entries indexed by the subsets $\{i_1, \ldots, i_{d-1}\}$, 
(\ref{rel}) shows that $Q\in I_n$ if and only if the vector of coefficients ${\bf b}$ is in the vector space spanned by the columns of $\mathcal C$. We have $\displaystyle \left(\begin{array}{c} n\\ d-3\\ \end{array}\right) < \left(\begin{array}{c} n\\ d-1\\ \end{array}\right)$, and therefore there are certain linear equations (with coefficients depending only on $n$ and $d$ ) that are satisfied among the $b_{i_1\cdots i_{d-1}}$.

On the other hand, we can view the coefficients of $Q$ as polynomials in the variables $a_1, \ldots, a_n$. Note that $b_{i_1\cdots i_{d-1}}$ has degree 2 in the variables $a_{i_1}, \ldots, a_{i_{d-1}}$, and degree 1 in all other variables. Therefore, $b_{i_1\cdots i_{d-1}}$ are linearly independent as polynomials in $a_1, \ldots, a_n$.
\end{proof}

We now discuss how the failure of injectivity proved in Theorem (\ref{square injectivity failure}) relates to Conjecture 2 from \cite{CriQ-Lun-Nen}:
\begin{conjecture}[\cite{CriQ-Lun-Nen}, Conjecture 2]\label{ext conj}
Let $\ell_1, \ell_2$ be general linear forms in $P_n$.
Then
$$
\mathrm{dim}\left(\left(\frac{P_n}{(x_1^2, \ldots, x_n^2, \ell_1^2, \ell_2^2)}\right)_d\right)=a_{n, d}
$$
where $a_{n, d}$ is the number of lattice paths from $(0, 0)$ to $(n+2-2s, n+2)$ with moves of the type $(i, j)\to (i+1, j+1)$ (right move) or $(i, j)\to (i-1, j+1)$ (left move) such that the first and the last move are to the right, and the path does not cross the lines $x=0$ and $x=n+2-2d$.
\end{conjecture}

The conjecture is proved in \cite{CriQ-Lun-Nen} for $n\ge 3d-3$. As a corollary of Theorem (\ref{square injectivity failure}), we are able to prove that the conjecture holds for $n=3d-4$, which is the first case of interest.
\begin{corollary}\label{tt cor}
    Conjecture (\ref{ext conj}) holds for $n=3d-4$.
\end{corollary}
\begin{proof}
Let $P_{n+1}$ denote the polynomial ring $\F[x_1,\dots,x_{n+1}]$ and $I_{n+1}:=\big(x_1^2,\dots,x_{n+1}^2, (x_1+\cdots+x_{n+1})^2\big)$. Via a change of coordinates, we can take 
$\ell_1=x_1+\cdots + x_n$.
Consider the map $\phi_{n+1, d}$ given by multiplication by a general form $\ell $ from $(P_{n+1}/I_{n+1})_{d-1}$ to $(P_{n+1}/I_{n+1})_d$.
The cokernel of this map is isomorphic to $\displaystyle \left(\frac{P_n}{(x_1^2, \ldots, x_n^2, \ell_1^2, \ell_2^2)}\right)_d$. Therefore,
$$
\mathrm{dim}\left(\left(\frac{P_n}{(x_1^2, \ldots, x_n^2, \ell_1^2, \ell_2^2)}\right)_d\right)=
\mathrm{dim}\left(\left(\frac{P_{n+1}}{I_{n+1}}\right)_d\right)-
\mathrm{dim}\left(\left(\frac{P_{n+1}}{I_{n+1}}\right)_{d-1}\right)+\mathrm{dim}(\mathrm{Ker}(\phi_{n+1, d})).
$$
On the other hand, the proof of Proposition 7 in \cite{CriQ-Lun-Nen} shows that 
$$
a(n, d)=\left(\begin{array}{c} n\\ d\\ \end{array}\right)-2\left(\begin{array}{c} n\\ d-2\\ \end{array}\right) +\left(\begin{array}{c} n\\ d-4\\ \end{array}\right) + T(n,d)
$$
where we define $T(n, d)$ to be the number of paths that cross the line $x=n+2-2d$ first and then also cross the line $x=0$ later.
In view of Remark (\ref{HF of ACI}), Theorem 8 in \cite{CriQ-Lun-Nen} implies 
\begin{equation}\label{tt}T_{n, d}\ge \mathrm{dim}(\mathrm{Ker}(\phi_{n+1, d})),
\end{equation} and so the conjecture is equivalent to equality in (\ref{tt}).

We claim that when $n=3d-4$, $T_{n, d}=1$. To prove this, we argue along the same lines as in the proof of Proposition 7 in \cite{CriQ-Lun-Nen}. Indeed, a path that crosses first $x=n+2-2d$ and then $x=0$ consists of three segments; the first segment from $(1, 1)$ to $(n+3-2d, k_1)$ (the first time when the line $x=n+2-2d$ is crossed) the second segment from $(n+3-2d, 1+k_1)$ to $(-1, k_1+k_2)$ (the first time when the line $x=0$ is crossed) and the third segment from $(-1, k_2)$ to $(n+1-2d, k_1+k_2+k_3)$, where $k_1+k_2+k_3=n$ (the total number of steps). If the first segment has $s_1$ steps to the left, then $n+2-2d=k_1-2s_1$. If the second segment has $s_2$ steps to the right, then $n+4-2d=k_2-2s_2$. Finally, if the third segment  has $s_3$ steps to the left, then $n+2-2d=n-2s_3$.
Adding the three equations above and recalling $k_1+k_2+k_3=n$, we obtain $3n+8-6d=n-2(s_1+s_2+s_3)$. This implies $s_1=s_2=s_3=0$, and therefore such a path must consist of only steps to the right for the first segment, only steps to the left for the second segment, and only steps to the right for the third segment. In particular, only one such path is possible.

On the other hand, Theorem (\ref{square injectivity failure}) 
shows that when $n\le 3d-4$, $\mathrm{dim}(\mathrm{Ker}(\phi_{n+1,d})) \ge 1$. When this is paired with (\ref{tt}), we have
\begin{equation}
    1=T_{n, d}\ge \mathrm{dim}(\mathrm{Ker}(\phi_{n+1, d}))\ge 1
\end{equation}
and therefore equality follows.
\end{proof}
\begin{remark}
\normalfont{The next case of interest for Conjecture (\ref{ext conj}) is $n=3d-5$. As in the proof of Corollary (\ref{tt cor}), one can see that for this case, a path that crosses first $x=n+2-2d$ and then $x=0$ has exactly one step in one of the three segments in the direction opposite to the direction of the segment. Since this step can occur at any one of the $n$ steps, we have $T(n, d)=n$. In order to establish that Conjecture (\ref{ext conj}) holds for this case, one would need to be able to produce $n$ linearly independent elements in the kernel of $\phi_{n+1, d}$.
A potential way to accomplish this is by choosing various specializations of the one element that is known to exist in the kernel of $\phi_{n+2, d}$.}
\end{remark}

\end{document}